\documentclass[11pt, a4paper]{amsart}
\usepackage{a4wide}

\usepackage[pdftex,
            pdftitle={Fast generation of isotropic Gaussian random fields on the sphere},
            pdfauthor={P.E. Creasey and A. Lang},
            bookmarksopen,
            colorlinks,
            linkcolor=black,
            urlcolor=black,
	   			 	citecolor=black
]{hyperref}

\usepackage{graphicx}
\usepackage{algorithm}
\usepackage{algpseudocode}

\newcommand{\mmax}{m_{\max}}

\newcommand{\R}{\mathbb{R}}
\newcommand{\C}{\mathbb{C}}
\newcommand{\N}{\mathbb{N}}
\newcommand{\Z}{\mathbb{Z}}
\DeclareMathOperator{\E}{\mathbb{E}}
\newcommand{\IS}{\mathbb{S}}
\newcommand{\IA}{\mathbb{A}}

\DeclareMathOperator{\Cov}{Cov}

\let\Re\relax
\DeclareMathOperator{\Re}{Re}
\let\Im\relax
\DeclareMathOperator{\Im}{Im}

\newcommand{\gd}{\delta}

\newcommand{\gk}{\kappa}

\newcommand{\gD}{\Delta}

\newcommand{\gO}{\Omega}

\newcommand{\cD}{\mathcal{D}} %Domains of Operators

\newcommand{\cL}{\mathcal{L}}

\newcommand{\cN}{\mathcal{N}}

\newcommand{\cS}{\mathcal{S}}

\newcommand{\cU}{\mathcal{U}}

\newcommand{\cY}{\mathcal{Y}}

\newcommand{\dd}{\mathrm{d}}

\newtheorem{lemma}{Lemma}[section]

\theoremstyle{remark}

\theoremstyle{definition}

\begin{document}

\title[Fast generation of isotropic GRFs on the sphere]{Fast generation of isotropic Gaussian random fields on the sphere}

\author[P.E.~Creasey]{Peter~E.~Creasey} \address[Peter~E.~Creasey]{\newline Department of Physics and Astronomy
\newline University of California, Riverside
\newline California 92507, USA.} \email[]{peter.creasey@ucr.edu}

\author[A.~Lang]{Annika Lang} \address[Annika Lang]{\newline Department of Mathematical Sciences
\newline Chalmers University of Technology \& University of Gothenburg
\newline 412 96 G\"oteborg, Sweden.} \email[]{annika.lang@chalmers.se}

\thanks{
\emph{Acknowledgement.}
PEC would like to thank Ben Lowing and AL the members of the KAW project ``Stochastics for big data and big systems'' for fruitful discussions.
The work of the AL was supported in part by the Knut and Alice Wallenberg foundation
and the Swedish Research Council under Reg.~No.~621-2014-3995. 
}

\date{\today}

\begin{abstract}
The efficient simulation of isotropic Gaussian random fields on the unit sphere is a task encountered frequently in numerical applications. A fast algorithm based on Markov properties and fast Fourier Transforms in 1d is presented that generates samples on an $n \times n$ grid in $\operatorname{O}(n^2 \log n)$. Furthermore, an efficient method to set up the necessary conditional covariance matrices is derived and simulations demonstrate the performance of the algorithm. An open source implementation of the code has been made available at \url{https://github.com/pec27/smerfs}.
\end{abstract}
\maketitle

\section{Introduction}

Methods for modelling spatially distributed data with isotropic random fields occur in many areas, including astrophysics \cite{Bardeen_1986, Hoffman_1991}, geophysics \cite{Rue_Held_2005}, optics \cite{Mendez_1987}, image processing \cite{Cohen_1991} and computer graphics \cite{Miller_1986}. The building block for almost all of which is the Gaussian random field in~$\mathbb{R}^d$, although many applications require realisations in other geometries, the most important of which may be the unit sphere, $\IS^2$, especially for geophysics and astrophysics, which often use spherical data.  

The computational complexity of realising an $n \times n$ lattice of points of a Gaussian random field in~$\R^d$ depends considerably upon the structure of the covariance function. In the worst (non-isotropic) case we must calculate the covariance between each point and perform a Cholesky decomposition of the resulting matrix, at a cost of $\operatorname{O}(n^6)$ operations \cite{Rue_Held_2005}. Fortunately, the majority of Gaussian random fields we are interested in are isotropic (i.e.\ the covariance depends only on the geodesic distance between points), allowing us to perform a spectral decomposition into harmonic (Fourier) components. In the periodic case on~$\R^2$ (i.e.\ a torus) with a regular lattice, this provides a dramatic computational improvement due to the existence of the fast Fourier transform (FFT) algorithm, which allows the realisation to be computed in $\operatorname{O}(n^2 \log n^2)$ operations.

On~$\IS^2$, similar results for isotropic Gaussian random fields also apply, i.e., we can perform a spectral decomposition into the spherical harmonic functions. Unfortunately the corresponding spherical transform algorithm has fewer symmetries to exploit, and for a general $n \times n$ grid on the sphere, $\operatorname{O}(n^4)$ steps are required for an isotropic covariance function. 
If one chooses an iso-latitude system such as the equidistant cylindrical projection (i.e.\ a regular grid in $\theta$ and $\phi$, the inclination and azimuthal angles) or the HEALPix discretisation (\cite{Gorski_2005}, as used for the cosmic microwave background data of WMAP) one can exploit the Fourier element of the azimuthal angle to reduce this to $\operatorname{O}(n^3 \log n)$ (see e.g.~\cite{Driscoll_1994}).

As one moves to higher and higher resolutions, however, it becomes even more attractive to find algorithms with superior scaling. 
One such possibility is to consider Gaussian Markov random fields (hereafter GMRFs) (e.g.\ \cite{Whittle_1954, McKean_1963, Pitt_1971, Tewfik_1991, Moura_1997, Vats_Moura_2011}) which extend the Markov property to higher dimensions and orders.
The essential property of these fields is that the probability measure for a conditional realisation of these fields depends upon only the boundary of the constraining set (and its derivatives) rather than the full set.
Such a property allows one to construct a realisation recursively, i.e.\ to iteratively expand the realisation around its boundary, a process sometimes described as a telescoping representation~\cite{Vats_Moura_2011}.

The requirement for an isotropic GMRF to obey this property is that the covariance function is the Green's function of an operator that is a polynomial in the Laplacian (with appropriate coefficients to ensure strong ellipticity, see e.g.~\cite{Moura_1997}). 
This turns out to describe rather a large class of covariance functions~\cite{Tewfik_1991}.

The contribution of this paper is an algorithm that generates samples on an $n \times n$ grid of the sphere in $\operatorname{O}(n^2 \log n)$ once the conditional covariance matrices are computed. This is achieved by decomposing the field into 1d Gaussian Markov random fields. These can be sampled together with the derivatives point by point and then transformed to an isotropic Gaussian random field on the unit sphere by FFT.

The paper is structured as follows: In Section~\ref{sec:theory} we derive the decomposition of an isotropic Gaussian random field into 1d GMRFs via Fourier transforms and compute the conditional covariance matrices. A computationally efficient method to set up the conditional covariance matrices is presented in Section~\ref{sec:eff_cov_comp}. We collect the results of the previous sections in Section~\ref{sec:alg}, where we present the algorithms explicitly. Finally, the performance and convergence of the introduced algorithm is demonstrated in a simulation example in Section~\ref{sec:simulation}. Our implementation is available online at \url{https://github.com/pec27/smerfs}.

\section{Decomposition of isotropic Gaussian random fields into 1d Gaussian Markov random fields}\label{sec:theory}

Let us assume that $T$ is a zero mean
$2$-weakly isotropic Gaussian random field (GRF for short) on the unit sphere~$\IS^2$ in~$\R^3$, i.e.\ on
\begin{equation*}
 \IS^2 := \{ x \in \R^3 : \|x\| = 1\},
\end{equation*}
where $\| \cdot \|$ denotes the Euclidean norm. Then $T$ admits an expansion with respect to the \emph{surface spherical harmonic functions} 
$\cY := (Y_{\ell m}, \ell \in \N_0, m=-\ell, \ldots, \ell)$ 
as mappings 
$Y_{\ell m}: [0,\pi] \times [0,2\pi) \rightarrow \C$, 
which are given by
  \begin{equation*}
   Y_{\ell m}(\theta, \phi)
    := \sqrt{\frac{2\ell + 1}{4\pi}\frac{(\ell-m)!}{(\ell+m)!}} P_{\ell m}(\cos \theta) e^{im\phi}
  \end{equation*}
for $\ell \in \N_0$, $m = 0,\ldots, \ell$, and $(\theta,\phi) \in [0,\pi] \times [0,2\pi)$ and by
\begin{equation*}
   Y_{\ell m}
    := (-1)^m \overline{Y_{\ell -m}}
\end{equation*}
for $\ell \in \N$ and $m=-\ell, \ldots,-1$.
Here $(P_{\ell m}, \ell \in \N_0,m=0,\ldots,\ell)$ denote the \emph{associated Legendre polynomials} which are given by
\begin{equation*}
   P_{\ell m}(\mu)
    := (-1)^m (1-\mu^2)^{m/2} \frac{\partial^m}{\partial \mu^m} P_\ell(\mu)
\end{equation*}
for $\ell \in \N_0$, $m = 0,\ldots,\ell$, and $\mu \in [-1,1]$,
where 
$(P_\ell, \ell \in \N_0)$ are the \emph{Legendre polynomials} 
given by Rodrigues' formula (see, e.g.~\cite{Szego})
\begin{equation*}%\label{eq:LegPol}
P_\ell(\mu)
:= 
2^{-\ell} \frac{1}{\ell!} \, \frac{\partial^\ell}{\partial \mu^\ell} (\mu^2 -1)^\ell
\end{equation*}
for all $\ell \in \N_0$ and $\mu \in [-1,1]$.
Associated Legendre polynomials with negative $m=-\ell,\ldots,-1$ satisfy
  \begin{equation*}
   P_{\ell m} = (-1)^{|m|} \frac{(\ell - |m|)!}{(\ell + |m|)!} P_{\ell -m}.
  \end{equation*}
This expansion of~$T$ is given by (see, e.g.~\cite[Corollary~2.5]{LS15})
  \begin{equation*}%\label{eq:KL_T}
   T
    = \sum_{\ell=0}^\infty \sum_{m=-\ell}^\ell a_{\ell m} Y_{\ell m},
  \end{equation*}
where $\IA := (a_{\ell m}, \ell \in \N_0, m= - \ell, \ldots, \ell)$ 
is a sequence of complex-valued, centred, Gaussian random variables 
with the following properties:
\begin{enumerate}
 \item $\IA_+ := (a_{\ell m}, \ell \in \N_0, m = 0,\ldots,\ell)$ is a  
       sequence of independent, complex-valued Gaussian random variables.
 \item The elements of $\IA_+$ with $m>0$ satisfy  
       $\Re a_{\ell m}$ and $\Im a_{\ell m}$ independent and $\cN(0,C_\ell/2)$ distributed.
   \item 
   The elements of $\IA_+$ with $m=0$ are real-valued and
$\cN(0,C_\ell)$ distributed.
   \item\label{eq:cond_al-m} The elements of~$\IA$ with $m <0$ are deduced from those of $\IA_+$
    by the formulae
    \begin{equation*}
    a_{\ell m} = (-1)^m \, \overline{a_{\ell-m}}.
    \end{equation*}
  \end{enumerate}
Here $(C_\ell, \ell \in \N_0)$ is called the \emph{angular power spectrum}.

In what follows let us consider the case that there exist $\gk_i \in \C$
such that
\begin{equation}\label{eq:Cl}
 C_\ell := \prod_{i=1}^M (\gk_i + \ell(\ell+1))^{-1}
\end{equation}
for all $\ell \in \N_0$, i.e.\ $C_\ell^{-1}$ is an eigenvalue of
\begin{equation*}
 \cL := \prod_{i=1}^M (\gk_i - \gD_{\IS^2})
\end{equation*}
with corresponding eigenfunctions $(Y_{\ell m}, m = -\ell,\ldots,\ell)$,
where $\gD_{\IS^2}$ denotes the \emph{spherical Laplacian} (also known as the \emph{Laplace--Beltrami operator}).
In the spirit of~\cite{LRL11}, $T$ is the solution of the stochastic partial differential equation
\begin{equation*}
 \cL^{1/2} T = W,
\end{equation*}
where $W$ is white noise which admits the formal Karhunen--Lo\`eve expansion
\begin{equation*}
 W = \sum_{\ell=0}^\infty \sum_{m=-\ell}^\ell \eta_{\ell m} Y_{\ell m}.
\end{equation*}
Here $(\eta_{\ell m},\ell \in \N, m= 1, \ldots, \ell)$ is a set of independent, complex-valued standard normally distributed random variables independent of the real-valued standard normally distributed random variables $(\eta_{\ell 0}, \ell \in \N_0)$. For $m<0$ the same relations as in Condition~\eqref{eq:cond_al-m} hold.
For fixed $m \in \Z$, define the random field $g_m: \gO \times [-1,1] \rightarrow \R$ by
\begin{equation*}
 g_m(z) := \frac{1}{2\pi} \int_0^{2\pi} T(\theta, \phi) e^{-i m \phi} \, \dd \phi
\end{equation*}
for $z := \cos \theta \in [-1,1]$.
We observe that
 \begin{equation*}
  T
    = \sum_{\ell=0}^\infty \sum_{m=-\ell}^\ell a_{\ell m} Y_{\ell m}
    = \sum_{m = -\infty}^\infty \sum_{\ell \ge |m|} a_{\ell m} Y_{\ell m}
 \end{equation*}
and obtain for any $m_0 \in \Z$
 \begin{equation*}
  g_{m_0}(z)
    = \sum_{m = -\infty}^\infty \sum_{\ell \ge |m|} a_{\ell m} L_{\ell m}(z) 
	  \frac{1}{2\pi} \int_0^{2\pi} e^{im\phi} e^{-i m_0 \phi} \, \dd \phi
    = \sum_{\ell \ge |m_0|} a_{\ell m_0} L_{\ell m_0}(z),
 \end{equation*}
where we set
  \begin{equation*}
  L_{\ell m}(z)
    := \sqrt{\frac{2\ell + 1}{4\pi}\frac{(\ell-m)!}{(\ell+m)!}} P_{\ell m}(z).
 \end{equation*}
Moreover, we obtain for $m_0 < 0$ the relation
 \begin{align*}
  g_{m_0}
    & = \sum_{\ell \ge |m_0|} a_{\ell m_0} L_{\ell m_0}
    = \sum_{\ell \ge |m_0|} (-1)^{m_0} \, \overline{a_{\ell -m_0}} 
		    \sqrt{\frac{2\ell + 1}{4\pi}\frac{(\ell-m_0)!}{(\ell+m_0)!}}
		    (-1)^{|m_0|} \frac{(\ell + m_0)!}{(\ell - m_0)!} P_{\ell -m_0}\\
    & = \sum_{\ell \ge |-m_0|} \overline{a_{\ell -m_0}} L_{\ell -m_0}
    = \overline{g_{-m_0}}.
 \end{align*}
Since $g_{m_0}$ is generated by a sum of centred Gaussian random variables, it is clear that $g_{m_0}$ is centred Gaussian.
In the following lemma, we show that they are independent for positive~$m$ and compute the covariance.

\begin{lemma}\label{lem:cov_gm}
 The sequence $(g_m, m \in \N_0)$ consists of pairwise independent centred Gaussian random fields on~$[-1,1]$ with covariance given by
 \begin{equation*}
  C_{g_m}(z_1,z_2)
    := \E(\overline{g_m(z_1)} g_m(z_2)) 
    = \sum_{\ell \ge |m|} C_\ell L_{\ell m}(z_1) L_{\ell m}(z_2), 
 \end{equation*}
 while the functions $(g_m,m<0)$ with negative index are determined by the relation
 \begin{equation*}
  g_m = \overline{g_{-m}}.
 \end{equation*}
\end{lemma}
\begin{proof}
 We have already seen that $g_m$ is centred Gaussian and that $g_m = \overline{g_{-m}}$. It remains to show independence as well as to compute the covariance. Therefore let us fix $m_0, m_1 \in \Z$ and $z_1,z_2 \in [-1,1]$. By the previous computations it holds that
 \begin{align*}
  \E(\overline{g_{m_0}(z_1)} g_{m_1}(z_2))
    & = \sum_{\ell_0 \ge |m_0|} \sum_{\ell_1 \ge |m_1|} \E(\overline{a_{\ell_0 m_0}} a_{\ell_1 m_1} ) L_{\ell_0 m_0}(z_1) L_{\ell_1 m_1}(z_2)\\
    & = \sum_{\ell \ge \max\{|m_0|, |m_1|\}} \E(\overline{a_{\ell m_0}} a_{\ell m_1} ) L_{\ell m_0}(z_1) L_{\ell m_1}(z_2)\\
    & = \gd_{m_0 m_1} \sum_{\ell \ge |m_0|} C_\ell L_{\ell m_0}(z_1) L_{\ell m_1}(z_2)
 \end{align*}
 using the properties of the random variables in~$\IA$, where $\gd_{m_0 m_1} = 1$ for $m_0 = m_1$ and $0$ else.
 Similarly we obtain
 \begin{equation*}
  \E(g_{m_0}(z_1) g_{m_1}(z_2))
    = \gd_{m_0 (- m_1)} \sum_{\ell \ge |m_0|} C_\ell L_{\ell m_0}(z_1) L_{\ell m_1}(z_2),
 \end{equation*}
 which
finishes the proof, since uncorrelated Gaussian random variables are independent.
\end{proof}

We define the operators
\begin{equation*}
 \cL_m := \frac{\dd}{\dd z} (1-z^2) \frac{\dd}{\dd z} - \frac{m^2}{1-z^2}
\end{equation*}
and
\begin{equation*}
 \cD_m := \prod_{i=1}^M (\gk_i - \cL_m).
\end{equation*}
Let us recall the associated Legendre differential equations for $\ell \in \N_0$ and $m=-\ell,\ldots,\ell$ (see, e.g.,~\cite{AMS})
\begin{equation*}
 \frac{\dd}{\dd z} \left((1-z^2) \frac{\dd}{\dd z}u(z)\right) + \left(\ell(\ell+1) - \frac{m^2}{1-z^2}\right) u(z) = 0,
\end{equation*}
which is solved by the associated Legendre polynomials. 
Therefore we obtain that
\begin{equation*}
 \cD_m L_{\ell m}
    = \prod_{i=1}^M (\gk_i + \ell(\ell+1)) L_{\ell m}
    = C_\ell^{-1} L_{\ell m},
\end{equation*}
i.e.\ $g_m$ solves the stochastic  differential equation
\begin{equation*}
 \cD_m^{1/2} g_m = W_m,
\end{equation*}
where $W_m$ is white noise on~$[-1,1]$ 
and $\cD_m^{1/2}$ is a well-defined operator of order~$M$ since $\cD_m$ is of order~$2M$ with positive spectrum. 

For $\tau \in (-1,1)$ let us first consider the open sets $A_\tau = (\tau,1)$ with complements $A^c_\tau = (-1,\tau]$. Since white noise satisfies the strong Markov property, we obtain for all $z \in A_\tau$ that
\begin{equation*}
 P\left(\cD_m^{1/2} g_m(z)\,|\, \cD_m^{1/2} g_m(s), s \in A^c_\tau\right)
  = P\left(\cD_m^{1/2} g_m(z)\,|\, \cD_m^{1/2} g_m(\tau)\right).
\end{equation*}
and $g_m$ is generated by a stochastic differential equation as in~\cite[Eqn.~(2.1.21)]{R82}. However, $g_m$ does not satisfy it but it follows from~\cite{R82} that the $M$-dimensional vector of derivatives up to order~$M-1$,
\begin{equation*}
g_m^{(0\ldots M-1)} = (g_m^{(0)}, g_M^{(1)}, \ldots, g_m^{(M-1)})^T
\end{equation*}
have the property. In other words, we obtain for all $z \in A_\tau$
\begin{equation*}
 P\left(g_m^{(0\ldots M-1)}(z)\,|\, g_m^{(0\ldots M-1)}(s), s \in A^c_\tau\right)
  = P\left(g_m^{(0\ldots M-1)}\,|\, g_m^{(0\ldots M-1)}(\tau)\right).
\end{equation*}
Therefore it is sufficient to know $g_m^{(p)}(\tau)$ for $p= 0,\ldots, M-1$ when generating a sample value for any $z > \tau$ where the random field is already constructed for all $z < \tau$. Since we are considering Gaussian fields which have Gaussian derivatives, it is sufficient to know the mean and the covariance between any two points and derivatives up to order~$M-1$.

Since all derivatives are well-defined, we obtain the covariance functions, which we will refer to as \emph{cross covariances},
\begin{equation*}
 C_{g_m,p,q}(z_1,z_2) 
  := \E\left(\overline{g_m^{(p)}(z_1)} g_m^{(q)}(z_2)\right)
  = \frac{\partial^{p+q}}{\partial z_1^p \partial z_2^q} C_{g_m}(z_1,z_2).
\end{equation*}
These derivatives can analytically be calculated with the identities of the associated Legendre polynomials.
We will discuss a computational efficient method in Section~\ref{sec:eff_cov_comp}. Let us set
\begin{equation*}%\label{eq:Jmatrix}
 J^m_{pq}(z_1,z_2) := C_{g_m,p,q}(z_1,z_2)
\end{equation*}
such that $J^m(z_1,z_2)$ is the matrix consisting of all covariances between different derivatives at $z_1$ and $z_2$. The conditional distribution of $g_m^{(0\ldots M-1)}(z_2)$ given $g_m^{(0\ldots M-1)}(z_1)$ is known to be Gaussian with conditional mean
\begin{align}\label{eq:trans_matrix}
 \begin{split}
  \E\left(g_m^{(0\ldots M-1)}(z_2)|g_m^{(0\ldots M-1)}(z_1)\right)
    & = J^m(z_1,z_2) (J^m)^{-1}(z_1,z_1)g_m^{(0\ldots M-1)}(z_1)\\
    & =: A^m(z_1, z_2) g_m^{(0\ldots M-1)}(z_1)
 \end{split}
\end{align}
and conditional covariance matrix
\begin{align}\label{eq:innov_matrix}
 \begin{split}
 \Cov\left(g_m^{(0\ldots M-1)}(z_2)|g_m^{(0\ldots M-1)}(z_1)\right)
  & = J^m(z_2, z_2) - J^m(z_1, z_2)  (J^m)^{-1}(z_1,z_1) (J^m)^T(z_1,z_2)\\
  & = J^m(z_2, z_2) - J^m(z_1, z_2) (A^m)^T(z_1,z_2)\\
  & =: B^m (B^m)^T(z_1,z_2).
 \end{split}
\end{align}

Using the values $g_m^{(0\ldots M-1)}(z_1)$ at $z_1$ for non-negative~$m$, it is straightforward to sample the vector $g_m^{(0\ldots M-1)}(z_2)$ from the conditional distribution by
\begin{equation}\label{eq:state_space}
g_m^{(p)}(z_2) = \sum_{q=0}^{M-1} A^m_{pq}(z_1, z_2)  g_m^{(q)}(z_1) + \sum_{q=0}^{M-1} B^m_{pq}(z_1, z_2) w^m_q,
\end{equation}
where the random variables $(w^m_q, q=0,\ldots,M-1)$ are independent complex-valued standard normally distributed for $m>0$, and independent standard normally distributed for $m=0$.
This construction is sometimes referred to as a state-space model~\cite{Brogan_1991}.

Note that we will also need an \emph{initial} sampling, and for this we choose the equator ($z=0$), i.e.
\begin{equation}\label{eq:state_initial}
g_m^{(p)}(0) = \sum_{q=0}^{M-1} B^m_{{\text{eq}}\;pq} w^m_q,
\end{equation}
where 
\begin{equation*}
  B^m_{\text{eq}} (B^m_{\text{eq}})^T := J^m(0, 0).
\end{equation*}

\section{Efficient covariance computation}\label{sec:eff_cov_comp}

We have shown in Lemma~\ref{lem:cov_gm} that
 \begin{equation*}
  C_{g_m}(z_1,z_2)
    = \E(\overline{g_m(z_1)} g_m(z_2))
    = \sum_{\ell \ge |m|} C_\ell L_{\ell m}(z_1) L_{\ell m}(z_2), \label{eq:bilinear}
 \end{equation*}
i.e.\ we obtained a representation in term of the associated Legendre polynomials. This allows us to compute it but from an algorithmic perspective, the computation is expensive for large~$m$ and~$\ell$ as the recurrence relations are unsuitable for numerical computation with limited precision. As such we find an alternative formulation in terms of Legendre functions that is numerically superior at even moderate-$m$ values.

Let us assume that the $\kappa_i$ are distinct in Eqn.~\eqref{eq:Cl} and have the representation  $\kappa_i = - \lambda_i (\lambda_i+1)$ with $\lambda_i \in \C \setminus \Z$ to avoid singularities. More specifically, for $\lambda_i \in \Z$, either $\ell = \lambda_i$ or $\ell = -\lambda-1$ is in~$\N_0$ and therefore $C_\ell$ is not well-defined.
A typical example is $M=2$ and $C_\ell = (a^2 +\ell^2 (\ell+1)^2)^{-1}$ with $a \in \R$, which satisfies the desired partial fraction decomposition with $\kappa_1 = i a$ and $\kappa_2 = -ia$.
Then the coefficients~$C_\ell$ in Eqn.~\eqref{eq:Cl} can be decomposed (via partial fractions) into
\begin{equation}\label{eq:C_ell_efficient}
C_\ell = \sum_{i=1}^M \frac{b_i}{\ell(\ell+1) - \lambda_i \left(\lambda_i+1 \right)}
\end{equation}
for some finite constants $b_i \in \C$, provided the $\kappa_i$ are distinct.

We note that the partial sums in Eqn.~(\ref{eq:bilinear}) can be written in the form
\begin{equation*}
 C_{g_m}(z_1, z_2) 
  = \frac{1}{2\pi}\sum_{i=1}^M b_i G^m_{\lambda_i} (z_1,z_2)
\end{equation*}
where 
\begin{equation*}%\label{eq:sumL}
G^m_{\lambda_i} (z_1,z_2) = 2\pi \sum_{\ell \ge |m|} \frac{ L_{\ell m}(z_1) L_{\ell m}(z_2)}{\ell(\ell+1) - \lambda_i \left(\lambda_i+1 \right)} , 
\end{equation*}
which is nothing more than the bilinear expansion for the Green's function \cite[Eqn.~(17.72)]{Riley06} for
the operator $\cD_m = \kappa_i -\cL_m$. This operator can be alternatively formulated as the solution to the differential equation
\begin{equation}\label{eq:generalized_Legendre_eqn} 
 \cD_m G^m_\lambda (x,y)  = \delta (x-y) .
\end{equation}
One eigenfunction of this operator with eigenvalue zero is the associated Legendre function $P^m_{\lambda_i}$, where $P^m_{\lambda_i}$ denotes the generalisation of $P_{\ell m}$ to non-integers $\lambda_i$.

Associated Legendre polynomials can be generalised to complex degree $\lambda$ and order $\mu$ as
\begin{equation*}
 P^\mu_\lambda (z) 
  = \frac{1}{\Gamma(1-\mu)} \left(\frac{1+z}{1-z}\right)^{\mu/2} \,_2F_1 \left(-\lambda, \lambda+1; 1-\mu; \frac{1-z}{2} \right),
\end{equation*}
via the (Gauss) hypergeometric function
\begin{equation*}
 _2F_1 \left(a,b;c; x \right) 
  = \sum_{n=0}^\infty \frac{(a)_n (b)_n}{(c)_n} \frac{x^n}{n!},
\end{equation*}
where $(x)_n$ denotes the rising Pochhammer symbol. The relevant case of $\mu$ being an integer (i.e.\ $\mu=m$) is only defined in the limit, which we find by (see \cite[15.3.3]{AMS})
\begin{equation*}
_2F_1 \left(a,b;c; x \right) = (1-x)^{c-a-b} \, _2F_1 \left(c-a,c-b;c; x \right) ,
\end{equation*}
and (see \cite[15.1.2]{AMS})
\begin{equation*}
\lim_{c\to -m} \frac{1}{\Gamma(c)} \, _2F_1 \left(a,b;c; x \right) 
  = \frac{(a)_{m+1} (b)_{m+1}}{(m+1)!} x^{m+1}\,  _2F_1 \left(a+m+1,b+m+1;m+2; x \right),
\end{equation*}
to give
\begin{equation}\label{eq:Legendre_hypergeometric}
P^m_\lambda (z) 
 = \lim_{\mu \to m} P^\mu_\lambda (z)
 = \frac{(-\lambda)_m (1+\lambda)_m}{m!} \left(\frac{1-z}{1+z}\right)^{m/2} 
      \,_2F_1 \left(-\lambda, \lambda+1; 1+m; \frac{1-z}{2} \right).
\end{equation}
This function is well-defined at $z=1$ with $P^m_\lambda (1) = \delta_{m0}$
and singular at $z=-1$ for all $m \in \N_0$ if $\lambda \notin \Z$.

For a second linearly independent 
solution to~\eqref{eq:generalized_Legendre_eqn}
we use $P^m_{\lambda_i} (-z)$, which is not linearly dependent since $\lambda+m$ is not an integer, see, e.g.~\cite[8.2.3]{AMS}. 
We can then construct the Green's function by using the Wronskian (e.g.\ \cite[14.2.3]{DLMF})
\begin{equation*}
 \mathcal{W} \left\{ P^m_\lambda (z) , P^m_\lambda (-z)  \right\} = \frac{2 \left(1-z^2\right)^{-1}}{\Gamma\left(\lambda+1-m\right) \Gamma\left(-\lambda-m\right)}
\end{equation*}
as
\begin{equation*}
G^m_{\lambda} (x,y) = \frac{1}{2} \Gamma\left(1+\lambda-m\right) \Gamma\left(-\lambda-m\right) \times 
    \begin{cases}
      P^m_\lambda(-x) P^m_\lambda(y) , & x\leq y,  \\
      P^m_\lambda(x) P^m_\lambda(-y) , & x\geq y,  
    \end{cases} 
\end{equation*}
which satisfies continuity, regularity at $x=-1$ and $x=1$, and by substitution
\begin{equation*}
 (-\lambda (\lambda+1) -\cL_m) G^m_\lambda (x,y) 
  = \delta (x-y).
\end{equation*}

In conclusion we obtain that 
\begin{align}\label{eq:Jmatrix_efficient}
  \begin{split}
 & J^m_{pq}(z_1,z_2) \\
  & \qquad = C_{g_m,p,q}(z_1,z_2)
  = \frac{1}{2\pi} \sum_{i=1}^M b_i G^m_{\lambda_i} (z_1,z_2)\\
  & \qquad = \frac{1}{4\pi} \sum_{i=1}^M b_i \Gamma\left(1+\lambda_i-m\right) \Gamma\left(-\lambda_i-m\right) \times 
    \begin{cases}
      (-1)^p P^{m+p}_{\lambda_i}(-z_1) P^{m+q}_{\lambda_i}(z_2) , & z_1 \leq z_2,  \\
      (-1)^q P^{m+p}_{\lambda_i}(z_1) P^{m+q}_{\lambda_i}(-z_2) , & z_1 \geq z_2.
    \end{cases}
 \end{split}
\end{align}
Substituting the associated Legendre functions by the hypergeometric functions as in~\eqref{eq:Legendre_hypergeometric}, i.e.\ setting
\begin{equation*}
 P^{m+p}_{\lambda_i} (z_1) 
 = \frac{(-\lambda_i)_{m+p} (1+\lambda_i)_{m+p}}{(m+p)!} \left(\frac{1-z_1}{1+z_1}\right)^{(m+p)/2} 
      \,_2F_1 \left(-\lambda_i, \lambda_i+1; 1+(m+p); \frac{1-z_1}{2} \right)
\end{equation*}
and similarly for $q$ and $z_2$, we derive our efficient computation method for the cross-covariance matrices $J^m$.

We observe that the above formula requires many re-computations of associated Legendre functions. To be more efficient, we transform the vector~$g_m^{(0\ldots M-1)}$ of derivatives up to order~$M-1$ to a vector of linear combinations of derivatives of increasing order. This does not change anything since we are in the end only interested in the random field itself but not in its derivatives. Therefore define for $m \in \N$ the family of operators $(\cU^q_m, q=0,\ldots,M-1)$ by
 \begin{equation*}
 \cU^q_m f := (-1)^q (1-x^2)^{(m+q)/2} \frac{\partial^q}{\partial x^q} \left[(1-x^2)^{-m/2} f\right],
 \end{equation*}
which satisfies $\cU^q_m P^m_\lambda = P^{m+q}_\lambda$. Thus setting 
 \begin{equation}\label{eq:Jmatrix_more_efficient}
 J^m_{pq}(x, y) := \cU^p_m \cU^q_m C_{g_m}(x,y)
 \end{equation}
is much easier to compute while the properties of the random field itself are not changed. Let us keep the notation $g_m^{(0\ldots M-1)}$ even if we use the modified matrices~$J^m$.

\section{Algorithms}\label{sec:alg}

Below we outline an algorithm for the construction of the isotropic GRF on~$\IS^2$. Since the sampling of the GRF is merely filtered white noise, in our tests the time is dominated by the computation of the (inhomogeneous) filter coefficients, and as such we have described the pre-computation in Alg.~\ref{alg:cov} which just has to be done once and the generation of samples in Alg.~\ref{alg:GRF}.

For $n \in \N$ let us fix a grid $z= (z_{-n},\ldots, z_0, \ldots, z_n)$ on $[-1,1]$ and $\phi = (\phi_0,\ldots, \phi_{2\mmax})$ for some $\mmax \in \N$. One possible symmetric example is to set $z_j = \sin\left(\frac{2\pi j}{2n+1}\right)$, such that
\begin{equation*}
 z_{-n} = \sin\left(-\frac{2n\pi}{2n+1}\right),
 \quad
 z_0 = 0,
 \quad
 z_n = \sin\left(\frac{2n\pi}{2n+1}\right),
\end{equation*}
and $\phi = \pi \, \mmax^{-1}(0, 1, 2,\ldots, 2\mmax-1)$ in $[0,2\pi)$.
Furthermore define the covariance operator $\cL$ by setting $\lambda = (\lambda_1,\ldots,\lambda_M)$ and $b = (b_1,\ldots,b_M)$ in~\eqref{eq:C_ell_efficient}.
\begin{algorithm}[htb!]
\caption{Pre-computation of covariance matrices}\label{alg:cov}
\begin{algorithmic}[0]
\Procedure{Covariance matrices}{$\mmax, z, b, \lambda$}
\For{$m=0,\ldots, \mmax+M-1$}
  \State Set up the hypergeometric functions ${}_2F_1$ of order~$m$ using $z$ and~$\lambda$.
  \For{$p=0,\ldots,M-1$}
    \State Set up the Pochhammer symbols for~$m$ and~$p$ using~$\lambda$.
  \EndFor
\EndFor
\For{$m=0,\ldots, \mmax$}
  \State Compute $J^m(z_0,z_0)$ using~\eqref{eq:Jmatrix_efficient} and the computed functions.
  \State Compute $B^m_{\text{eq}}$ from $J^m(z_0,z_0)$ (e.g.\ via Cholesky decomposition).
  \For{$i=-n+1,\ldots,n$}
    \State Compute $J^m(z_i,z_i)$ using~\eqref{eq:Jmatrix_efficient} or~\eqref{eq:Jmatrix_more_efficient} and the computed functions.
    \State Compute $J^m(z_{i-1},z_i)$ using~\eqref{eq:Jmatrix_efficient} or~\eqref{eq:Jmatrix_more_efficient} and the computed functions.
    \State Compute $A^m(z_{i-1},z_i)$ using~$J^m$ and~\eqref{eq:trans_matrix}.
    \State Compute $B^m(z_{i-1},z_i)$ using~$J^m$, $A^m$, \eqref{eq:innov_matrix} and, e.g.\ Cholesky decomposition.
  \EndFor
\EndFor
\State Return $A = (A^m, m=0,\ldots, \mmax)$, $B = (B^m, m=0,\ldots, \mmax)$, and $B_{\text{eq}} = (B^m_{\text{eq}}, m=0,\ldots, \mmax)$.
\EndProcedure
\end{algorithmic}
\end{algorithm}

Alg.~\ref{alg:cov} is devoted to the pre-computations that set up the covariance matrices that just have to be computed once and can be reused for sampling on the same grid. We remark that the algorithm can be optimised if the discrete grid~$z$ is symmetric around zero, i.e.\ if it satisfies $z_{-i} = -z_i$. Then $A^m$ and $B^m$ just have to be computed for all $z_i \ge 0$ and the negative values follow by symmetry. This is the case for the example grid~$z$ above.

\begin{algorithm}[htb!]
\caption{Gaussian random field generation}\label{alg:GRF}
\begin{algorithmic}[0]
\Procedure{Gaussian Random Field}{$\mmax, z,\phi, A, B, B_{\text{eq}}$}
\For{$m=0,\ldots, \mmax$}
  \State Generate $g_m^{(0\ldots M-1)}(z_0)$ using $B^m_{\text{eq}}$, random samples, and Eqn.~\eqref{eq:state_initial}.
\EndFor
\For{$i=1,\ldots,n$ and $i= -1,\ldots,-n$}
  \For{$m=0,\ldots, \mmax$}
    \State Generate $g_m^{(0\ldots M-1)}(z_i)$ using Eqn.~\eqref{eq:state_space}, random samples, $A^m$, and~$B^m$.
    \If{$m>0$}
      \State Set $g_{-m}^{(0)}(z_i) \gets \overline{g_m^{(0)}(z_i)}$.
    \EndIf
  \EndFor
\EndFor
  \State Compute $T(z,\phi)$ via discrete Fourier transform (e.g. via FFT for equi-spaced $\phi$) using
       \begin{equation*}
 	T(z_i, \phi_j) \gets \frac{1}{2\mmax\sqrt{2\pi}} \sum_{m=1-\mmax}^{\mmax} g_m(z_i) e^{i \phi_j m}
       \end{equation*}
\State Return $T$
\EndProcedure
\end{algorithmic}
\end{algorithm}

Samples of the GRF on the sphere are generated with Alg.~\ref{alg:GRF} based on the pre-computed families of matrices $A$, $B$, and~$B_{\text{eq}}$. Since the computational cost to generate white noise is negligible, the complexity of the algorithm reduces to $(2n+1)$ FFT in 1d, which makes it computationally fast. Therefore the algorithm is attractive if many samples have to be computed. An example are ice crystals, which can be modelled by lognormal random fields~\cite{NMF04}. For more information on the relation of lognormal and Gaussian random fields, the reader is referred to~\cite{LS15}.

\section{Simulation}\label{sec:simulation}

To illustrate the algorithm and show the performance of the presented algorithm, we show simulation results in this section which include the generation of a sample and the convergence of the covariance function with respect to the discretisation. 

Let us first show random field samples that the algorithm generates. Therefore we fix the angular power spectrum given by
\begin{equation}\label{eq:sample_spec}
C_\ell = \left(10 + \ell^2 (\ell+1)^2\right)^{-1}.
\end{equation}
The algorithm in Section~\ref{sec:alg} generates the one-dimensional random fields $(g_m, m=0,\ldots,\mmax)$ on a discrete grid first, which are then transformed to random fields on the sphere by FFT. To give the reader an idea of~$g_m$, we show in Fig.~\ref{fig:sample_walks} samples of $g_0(\cos \theta)$ and the real part of $g_5(\cos \theta)$ along with its standard deviation or RMS expectation $\E\left( \Re ( g_m(\cos \theta))^2 \right)^{1/2}$.
\begin{figure}
\centering
\includegraphics[width=\textwidth]{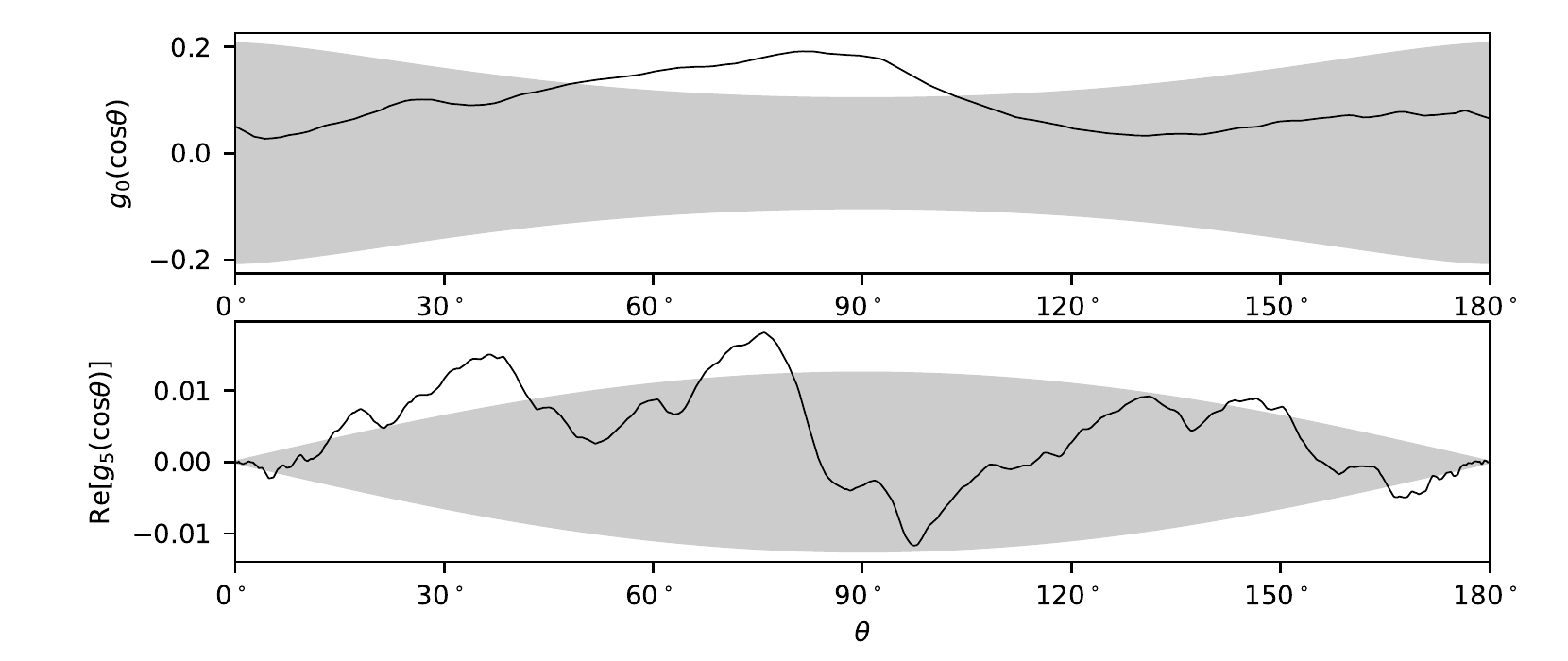}
\caption{Random walks $g_0$ and (the real part of) $g_5$ in the \emph{upper} and \emph{lower panels} respectively, for the angular power spectrum given in Eqn~(\ref{eq:sample_spec}). The \emph{black line} indicates the walks, whilst the \emph{grey shaded region} indicates the $\pm$ standard deviation ${\E}\left( \Re ( g_m )^2 \right)^{1/2}$.
}
\label{fig:sample_walks}
\end{figure}

The standard deviation is dependent on $\theta$ since the random field is of zero mean but not translation invariant.

Note that $g_5$ tends to zero at $\theta=0^\circ,180^\circ$. Since $T$ has a continuous first derivative (i.e.\ $M>1$, as it is in this case), then for all $m \neq0$ the standard deviation is zero for $\cos \theta = \pm 1$, i.e.\ the nonzero iso-latitude Fourier transforms must tend to zero near the poles to meet the smoothness criterion.
\begin{figure}
\centering
\includegraphics[width=\textwidth]{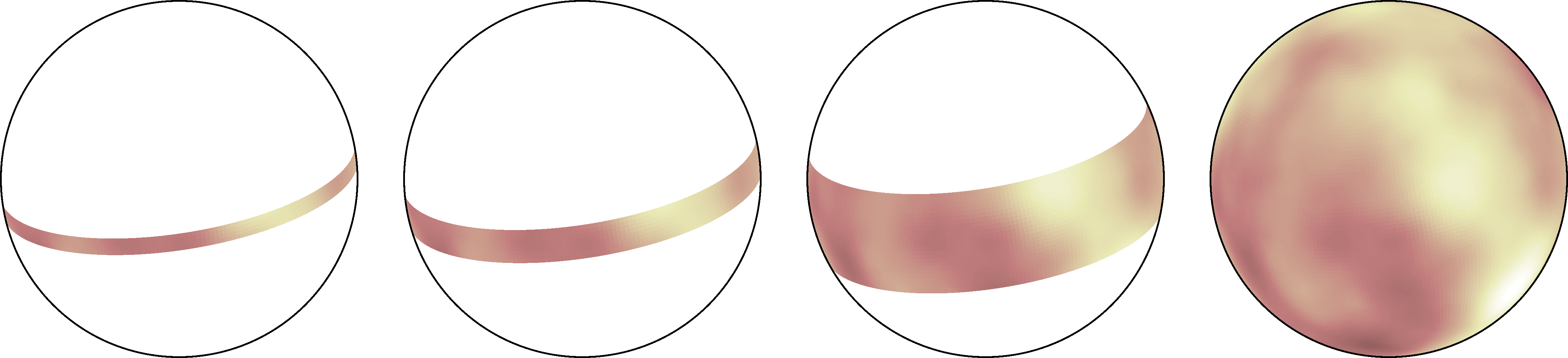}
\caption{\emph{Left} to \emph{right}, visualisation of the process of building the random field starting from the equator, i.e.\ at each iso-latitude line (constant $z$) we conditionally sample the $g_m(z)$ based on the previous line (including derivatives) and gradually work our way north and south.}
\label{fig:example_vis}
\end{figure}

The process of generating a sample is illustrated in Fig.~\ref{fig:example_vis}. 
We start with generating the random fields $g_m(0)$, which become after FFT the discrete sample of the random field at the equator. Conditionally, $g_m(z_{-1})$ and $g_m(z_1)$ are sampled and added to the picture after FFT. This process is continued until all of the sphere is covered with random numbers and the random field on the sphere is complete.

\begin{figure}
\centering
\includegraphics[width=\textwidth]{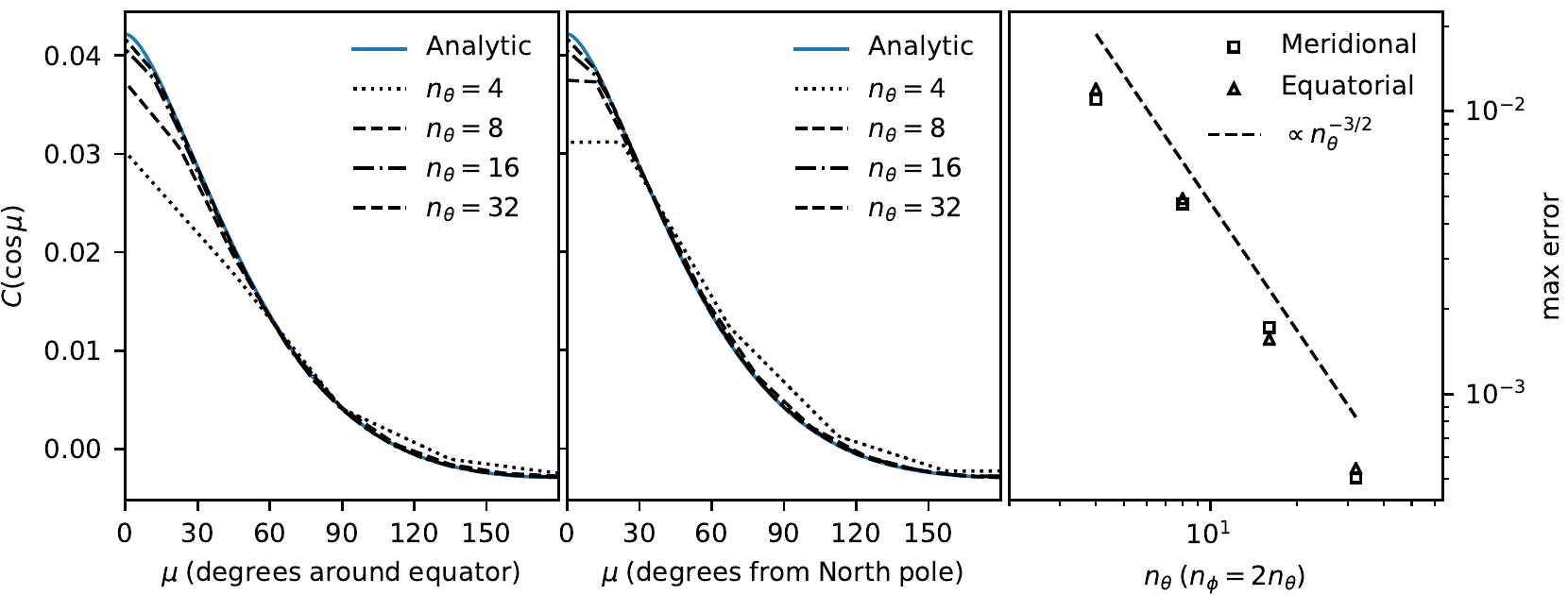}
\caption{Error analysis for an example GMRF. The \emph{left} and \emph{middle} panels estimate the covariance along equatorial and meridional lines respectively, and the analytic covariance in blue. The \emph{right panel} indicates the convergence of the error as a function of $n_\theta$.}
\label{fig:errors}
\end{figure}

A method to validate that a method for sampling a GRF is appropriate and that it converges is to compute the covariance function via the examination of many samples and compare it to the analytic one inferred from the angular power spectrum.
We have used the GRF with angular power spectrum from Eqn~(\ref{eq:sample_spec})
and corresponding covariance function
\begin{equation*}
C_T(x,y) = \sum_{\ell=0}^{\infty} \frac{2\ell+1}{4\pi} C_\ell P_\ell (\langle x,y\rangle_{\R^3}),
\end{equation*}
which is plotted in the left two panels of Fig.~\ref{fig:errors} as solid blue line. 

In order to analyse the errors of our method we have constructed filters at multiple resolutions of $n_\theta \in \{4,8,16,32\}$, keeping $\mmax = n_\theta$ and $n_\phi=2n_\theta$ (i.e.\ so the maximum angular separation between adjacent grid points is $\pi/n_\theta$). We generated $N:= 320,000$ samples $(T^{n_\theta}_j, j=1,\ldots, N)$ on all resolutions to estimate the covariance by
\begin{align*}
\overline{\Cov}(x,y) :=
  N^{-1} \sum_{j = 1}^{N}
      T^{n_\theta}_j(x) \cdot T^{n_\theta}_j(y)
\end{align*}
and to compare it to the theoretical one.
The dotted and dashed lines in the left two panels of Fig.~\ref{fig:errors} show the results for the different $n_\theta$. The left figure is based on equatorial evaluations while the middle one estimates the covariance along meridional lines.

The error was computed by taking on each grid the maximum over all grid points in a set~$\cS$, which was once the equator and once the meridian, of the
difference of the theoretical and statistical covariance, i.e.\ the error~$e^{n_\theta}$ was computed by
\begin{align*}
 e^{n_\theta}
  := \max_{x,y \in \cS}
      |\overline{\Cov}(x,y) - C_T(x,y)|.
\end{align*}
The results are shown in the right panel of Fig.~\ref{fig:errors}. The maximum error on the covariance function falls as $\mmax^{-3/2}$ (where $n_\theta=\mmax$) as can be seen in the error plot.

\bibliographystyle{siam}
\bibliography{paper} 

\end{document}